\documentclass[11pt,oneside,a4paper,english,reqno]{amsart}
\usepackage{amsmath}
\usepackage{mathrsfs}
\usepackage{amsthm}
\usepackage{amssymb}
\usepackage{todonotes}
\usepackage{hyperref}
\usepackage{babel}
\usepackage{times}

\usepackage{color}

\makeatletter

\newtheorem{thm}{Theorem}[section]
\newtheorem{prop}[thm]{Proposition}
\newtheorem{lem}[thm]{Lemma}
\newtheorem{cor}[thm]{Corollary}
\newtheorem{hyp}[thm]{Hypothesis}
\theoremstyle{definition}
\newtheorem{defin}[thm]{Definition}
\theoremstyle{remark}
\newtheorem{rem}[thm]{Remark}
\newtheorem{ex}[thm]{Example}

\numberwithin{equation}{section}
\makeatother

\newcommand{\cB}{\mathcal{B}}
\newcommand{\cF}{\mathcal{F}}
\newcommand{\cI}{\mathcal{I}}
\newcommand{\cL}{\mathcal{L}}
\newcommand{\cP}{\mathcal{P}}

\newcommand{\EE}{\mathbb{E}}
\newcommand{\II}{\mathbb{I}}
\newcommand{\PP}{\mathbb{P}}
\newcommand{\RR}{\mathbb{R}}

\newcommand{\WW}{\mathbb{W}}
\newcommand{\dd}{\mathop{}\!\mathrm{d}}

\newcommand{\abs}[1]{\left|#1\right|}

\newcommand{\pd}{\partial}
\newcommand{\ep}{\varepsilon}
\renewcommand{\ge}{\geqslant}
\renewcommand{\le}{\leqslant}

\allowdisplaybreaks

\usepackage{mathtools}
\let\oldtocsection=\tocsection
\let\oldtocsubsection=\tocsubsection
\let\oldtocsubsubsection=\tocsubsubsection
\renewcommand{\tocsection}[2]{\hspace{0em}\oldtocsection{#1}{#2}}
\renewcommand{\tocsubsection}[2]{\hspace{1em}\oldtocsubsection{#1}{#2}}
\renewcommand{\tocsubsubsection}[2]{\hspace{2em}\oldtocsubsubsection{#1}{#2}}
\numberwithin{thm}{section}

\allowdisplaybreaks

\title[Regularisation by noise]{Weak Existence for Degenerate Distribution Dependent SDEs with multiplicative Noise 
\\
- a pathwise regularization approach} 
\author{Fabian n. Harang \and Chengcheng Ling \and Peter H. C. Pang}
\keywords{pathwise regularisation-by-noise, McKean--Vlasov equation, weak solutions}
\subjclass[2020]{60H50,60H10,60H15,60L90,35K65,35K59}
\thanks{FH gratefully acknowledge the support of the Center for Advanced Studies (CAS) in Oslo, which funded the Signatures for Images project during the academic year 2023/2024, enabling the writing of this article. CL is supported by the Deutsche Forschungsgemeinschaft (DFG) - Projektnummer 563883019. PP is supported by the Research Cou. ncil of Norway project INICE, project no. 301538.}
\begin{document}
\begin{abstract} 
We establish the existence of weak solutions to a class of  distribution-dependent stochastic differential equations (DDSDEs) with possibly degenerate multiplicative noise and singular coefficients. 
Extending the weak existence techniques introduced by Bechtold \& Hofmanova \cite{bechtold} to a distribution-dependent framework, we utilize pathwise averaging and local-time decomposition methods to show how irregular noise effectively regularizes 
analytical challenges associated with degeneracies in stochastic systems.
\end{abstract}
\maketitle
\setcounter{tocdepth}{1}

\section{Introduction}
\label{sec:intro}

Dynamics subject to noise often exhibit a surprising phenomenon known as {\em regularisation by noise}, whereby singular or weakly well-posed equations regain (or improve) solvability once perturbed by irregular stochastic drivers. While such effects have been extensively studied in finite-dimensional stochastic differential equations (SDEs) with non-Lipschitz drifts \cite{MR336813, MR568986,  MR2820071,  MR2377011, gubicat, GG25, DGLL} etc, fewer works have tackled (McKean–Vlasov type) 
distribution-dependent SDEs (DDSDEs) under similarly rough forcing, and even fewer still allow for relaxed conditions on the coefficients in front of a multiplicative Brownian motion. The present paper establishes the existence of weak solutions to distribution-dependent SDEs driven by a multiplicative Brownian motion, in the setting where the distribution dependence is perturbed by a suitably irregular path, thus extending and combining both the recent results from \cite{bechtold} and \cite{HarangMayorcas}.

\subsection*{Motivation and contribution}

Systems of DDSDEs arise  in many contexts in 
mathematical modelling. 
In finance and economics, such equations are 
frequently used to represent individual
behaviour among a collective, describing the  influence on individual trajectories of 
expectations regarding the collective (see, for example, \cite{Wang2017,Wang2012,Nguyen2021,Bae2015}).

The present paper was motivated by the analysis of distribution dependent dynamics where each equation might interact with the distribution of all other equations. 
In particular, distribution dependent equations of the form 
\begin{equation}\label{eq:mckean}
    \dd x_t = b(t,x_t,\cL(x_t))\dd t + \sigma(t,x_t,\cL(x_t))\dd \beta_t, 
\end{equation}
where $\{\beta_t\}_{t\geq 0}$ is a $d$-dimensional Brownian motion, and $b:[0,T]\times \RR^n\times \cP(\RR^n)\rightarrow \RR^n$ and $\sigma:[0,T]\times \RR^n \times \cP(\RR^n)\rightarrow \RR^{n\times d}$ are sufficiently regular, and $\cP(\RR^n)$ denotes the space of probability measures over $\RR^n$. Here $\cL(x_t)$ denotes the law of the random variable $x_t$. 
Such equations arise naturally as generalisations of more concrete interacting particle systems of the form for $i=1,\ldots,N$
\[
\dd y_t^i = b\left(\frac{1}{N}\sum_{j\neq i} f_1(y_t^i-y_t^j)\right)+\sigma\left(\frac{1}{N}\sum_{j\neq i} f_2(y_t^i-y_t^j)\right)\dd \beta^i_t,
\]
where now $b$ and $\sigma $ are suitable functions, and $f_1$ and $f_2$ are used to describe the potential interaction between the particles \cite{Sznitman1991}. 
Under suitable regularity conditions (typically Lipschitz continuity with linear growth for all involved functions $b,\sigma,f_1,f_2$), one can prove so called propagation of chaos: This is the phenomenon whereby, as the number of interacting particles 
$N$ tends to infinity, the joint distribution of any fixed finite subset of particles converges to the product of identical one-particle laws satisfying a McKean–Vlasov equation of the form \eqref{eq:mckean}, rendering the particles asymptotically independent. 
In practice, one often chooses singular interaction kernels \(f_1\) and \(f_2\) (and nonlinear coefficients) \cite{hrz} to enforce strong repulsion \cite{BolleyCanizoCarrillo2011,HaLiu2009}—for example, in McKean–Vlasov flocking models of bird positions, using
\[
f_1(x)=f_2(x)=|x|^{-\gamma}
\]
generates sufficiently large repulsive drift and diffusion forces to robustly prevent collisions between trajectories.
However, such singular interaction is not Lipschitz continuous, and well-posedness of the dynamics thereby becomes an important problem.

A key ingredient in our approach is the pathwise averaging method first introduced by Catellier and Gubinelli \cite{gubicat} for  SDEs. Their main insight is that, under mild conditions on the noise, one can rewrite the drift terms as averaging operators (integrals against an “occupation measure”) that act as spatial mollifiers --- even if the nominal drift appears too singular for classical methods. More recently, one of the authors of the current paper together with Perkowski \cite{harang2020cinfinity}  refined this averaging argument further by decomposing it through local times associated with the driving noise sample paths.  The decomposition into the study of an occupation measure/local time reveals a powerful mechanism: rough noise trajectories --- suitably quantified via local-nondeterminism conditions --- end up conferring regularity to otherwise ill-posed ordinary differential equations (ODEs) via their convolution with singular drift functions. The pathwise regularisation by noise techniques has later been extended and explored in much detail and has now become a standard tool in the study of SDEs; see e.g., \cite{MR4594437, MR4498199, MR4404773, MR4528970, MR4488556, MR4635720}. 

More recently, inspired by the development of the stochastic sewing lemma \cite{Le20}, these techniques have further been developed in a mixed way, for instance  questions of existence and uniqueness of probabilistically weak solutions to SDEs has been studied by involving similar techniques (see e.g., \cite{bechtold,BechtoldHarang2024, BG, ALL23, BM,  ABLM}).  In these results, pathwise regularisation techniques are used in combination with classical methods from stochastic analysis, such as tightness and martingale arguments to prove strong or weak existence of solutions in regimes  not covered by ``pure" pathwise-regularisation-by-noise, especially when the perturbed path has particular structure and the framework requires more on the probabilistic properties for such path, for instance the singular equations driven by multiplicative noise \cite{butkovsky2023stochastic, DGLL} and weak solution theory \cite{BM}. While the stochastic sewing lemma provides powerful mathematical tools applicable in this context, it introduces additional technical complexity without necessarily enhancing the pathwise regularisation effect central to our analysis. Hence, in this work, we focus instead on the more direct pathwise averaging and local-time decomposition methods, which explicitly demonstrate the regularising effect of a rough signal in a streamlined manner.

To study pathwise regularisation of DDSDEs with non-Lipschitz coefficients driven by a  multiplicative Brownian motion, we will adapt and extend the ideas for weak existence of solutions to multiplicative classical SDEs developed in \cite{bechtold}.
 More specifically, we let $w:[0,T]\rightarrow \RR^k$ be  a continuous path which produces regularization and let $\{\beta_t\}_{t\in [0,T]}$ be a $d$-dimensional Brownian motion on a filtered probability space $(\Omega,\cF,\{\cF_t\}_{t\in [0,T]},\PP)$, and we
 consider the equation
\begin{equation}\label{eq:main intro}
    \dd x_t = b(t, F(\mathcal{L}(x_t)) - w_t)\dd t 
    	+ a(t, F(\mathcal{L}(x_t)) - w_t)\dd \beta_t, \quad x_0\in \RR^n. 
\end{equation}
Here $b:\RR_+\times\RR^k\rightarrow \RR^n, a:\RR_+\times\RR^k\rightarrow \RR^{n\times d}
    $, $\mathcal{L}(x_t)$ is the law of $x_t$ as above, 
$F: \mathcal{P}_p(\RR^n) \to \RR^k$ is a Lipschitz map 
from the space of probability measures on $\RR^n$ with 
$p$ moments under the Wasserstein $\mathbb{W}_p$ 
metric (see Definition \eqref{defin:wasserstein} below) 
to the finite dimensional space $\RR^k$. A useful example of such 
a map to keep in mind would be $F(\mathcal{L}(x_t)) = \EE f(x_t)$, for a fixed 
$f \in C^1(\RR^n;\RR^k)$. 
The path $w$ is deterministic and 
will later be assumed to possess certain regularising effects. Comparing with \eqref{eq:main intro} and \eqref{eq:mckean} we see that we have introduced an additional noise into the
dynamics, this relates to the well known idea that additive common noise cannot have
a regularising effect on singular particle systems, see \cite{DFV} for a discussion in one dimension. 
Our requirement that the dependence on 
$\mathcal{L}(x_t)$ occurs strictly via a map 
$F$ with finite-dimensional range reflects the 
constrain that we are not able now to deal with rough regularising paths in infinite dimensional space.

Because the drift and diffusion can depend on the evolving 
law of the solution, the classical well-posedness arguments 
used for pointwise SDEs are no longer applicable. 
Nonetheless, by leveraging the local-times-based version 
of the pathwise averaging method,
we show that even if 
the drift is distribution-dependent and insufficiently smooth, 
the noise can still restore enough regularity to guarantee the 
existence of solutions. We shall follow the proof strategies 
for classical SDEs developed by Hofmanova and Bechtold in 
\cite{bechtold} and later \cite{BechtoldHarang2024}, by employing tightness techniques for 
McKean–Vlasov equations—together with a Skorokhod 
representation argument—to pass to the limit in a suitable 
sequence of approximate solutions.

These techniques are 
extended and adapted to admit distribution dependent 
coefficients. By identifying the limiting process in a martingale 
sense, we establish existence of weak solutions to 
these singularly driven Mckean-Vlasov SDEs. As already motivated, models either from the physical or social sciences based on such equations often require singular or otherwise degenerate coefficients, making them difficult to analyse, and very often ill-posed. Our results  broadens 
the known scope of noise-induced regularisation 
to more sophisticated stochastic systems, 
underscoring that rough forcing can serve as a 
powerful antidote to degeneracies in a wide range 
of applications. 

In the remainder of this paper, we lay out the 
assumptions and state our main results in the 
subsection immediately following. In Section 
\ref{sec:notation_ss}, we briefly review 
notational conventions and introduce key tools 
to describe convergence of measures and 
roughness of the regularising path. 
In Section \ref{sec:stochastic integration}, 
we describe the mechanism of pathwise 
regularisation and apply this tool 
to establish well-posedness of an approximate 
system of DDSDEs. In Section \ref{sec:tightness}, 
we establish requisite tightness of laws of 
solutions to the approximating systems. 
Finally, in Section \ref{sec:lim} we conclude 
the limiting argument and establish weak 
existence of solutions to \eqref{eq:main intro}.

\subsection{Main results}

Throughout the paper, we make the following assumptions on the coefficients 
$a$, $b$, the regularising path $w$ and the map $F$:
\begin{hyp} \label{hyp:main} {\,}

    \begin{itemize}

    \item[(i)] For $a,b$ we require 
    $a, |a|^2, b\in L^\infty_tL^2_x \cap C^{\gamma_0}_t H^{-1}_x$ 
    with $\gamma_0\in(0,1)$.

    \item[(ii)] The map $F: \mathcal{P}_1(\RR^n) \to \RR^k$ is Lipschitz 
from the space of probability measures on $\RR^n$ equipped with
the Wasserstein $\mathbb{W}_1$ to the finite dimensional space $\RR^k$, that is 
 for $F \in {\rm Lip}(\mathcal{P}_1(E))$ and 
    $X\sim \mu \in \cP_p(E)$ and $Y\sim \nu\in \cP_p(E)$ then
    \begin{equation}\label{eq:1wasserstein_use}
        |F(\mu) - F(\nu)|\leq |F|_{\rm Lip} \WW_1(\mu,\nu) \le |F|_{\rm Lip} \EE |X - Y|.  
\end{equation}
    \item[(iii)] The regularising path $w$ is $\zeta_0$-locally non-deterministic with  $\zeta_0$ satisfying%
    $$
     \Big(2 + \frac{3k}2\Big) \zeta_0 < 1, \qquad
    \Big(1 + \frac{k}2\Big) \zeta_0 < \gamma_0,
    $$
    where $k$ is the dimension of the spatial variable of the path, and 
    $\gamma_0$ is the H\"older index of $a$ and $b$ in $H^{-1}_x$ from (i).
 \end{itemize}

\end{hyp}

Let us explain some of our assumptions which are partly 
implemented for clarity of presentation. As mentioned, 
we shall be using the local time to 
express functions of the rough driver $w$. The choice of 
the inclusion of the coefficients $a$ and $b$ in the 
space $L^\infty_t L^p_x \cap C^{\gamma_0}_t H^{-1}$ with $p = 2$ in 
(i) is a simplifying choice that allows us directly 
to apply known results on the spatial regularity of local times
\cite[Theorem 3.1]{harang2020cinfinity}. 
There are results on SDEs with additive noise and drift 
coefficients having more singular behaviour (see, e.g., \cite{harang2020cinfinity, MR4594437} 
and references therein). We believe that those regimes 
are attainable in our setting (with multiplicative noise). 
However, as our goal is not to improve upon those results 
for the drift coefficient in the setting of multiplicative 
noise, and we refrain from maximally relaxing assumptions 
on drift coefficients, in order to highlight key advancements 
related to the possibly degenerate diffusion coefficient. 
In the same vein, we further required $|a|^2$ to be in the 
same space as $a$ and $b$ as an assumption that will lighten 
the calculations presented below (though this will necessarily 
mean that $a \in L^\infty_t (L^2_x \cap L^4_x)$). 
The object $aa^T$ appears in the It\^o correction term when 
the It\^o formula is applied to $x_t$ in \eqref{eq:main intro}, 
and much of the calculations for $a$ and $b$ can now 
be directly repeated for $aa^T$. 

Finally, we require $F$ to be globally Lipschitz in order to control 
differences $F(\mu) - F(\nu)$ by the Wasserstein distance 
via \eqref{eq:1wasserstein_use}.

Under Hypothesis \ref{hyp:main}, our main results can be summarised as follows and we give the corresponding proofs in Section \ref{sec:lim}.
\begin{thm}\label{thm:main_result}
   If Hypothesis \ref{hyp:main} holds, then there exists a weak solution 
   $x \in L^1_\omega C^{\gamma_1/4}_t$ to \eqref{eq:main intro} with 
   \begin{align}\label{con.lem-gamme1}
    \frac{\gamma_1}2 > \Big(1 + \frac{k}2\big) \zeta_0.
    \end{align}
\end{thm}

\subsection{Notation and preliminaries}\label{sec:notation_ss}

In this section, we house the definitions of several technical 
notions of use throughout the paper. In particular, we state precisely 
the way that the regularising path is rough via an index of ``local 
non-determinism". We also review the key ``sewing lemma" that will be 
used repeatedly in this paper. Finally, we record several notational 
conventions we shall be employing.

\subsubsection*{\bf Notations and conventions}
  For any $N\in \mathbb{N}$ and $p\in [1,\infty]$, we denote by $L^p(\RR^d;\RR^N)$ the standard Lebesgue space; when there is no risk of confusion in the parameter $N$, we will simply write $L^p_x$ for short and denote by $\| \cdot\|_{L^p_x}$ the corresponding norm.
Similarly for the Bessel potential spaces $W^{\beta,p}_x=W^{\beta,p}(\RR^d;\RR^N)$, which are defined for $\beta\in\RR$, with corresponding norm
$$\|\varphi\|_{W_x^{\beta, p}}:=\|(\II-\Delta)^{\beta/2}\varphi\|_{L_x^{ p}};$$
$H^\beta_x:=W_x^{\beta, 2}$. For $\alpha\in [0,\infty)$, $C_x^\alpha =C^\alpha (\mathbb{R}^d;\RR^N)$ stands for the usual H\"older continuous function space, made of continuous bounded functions with  continuous and bounded derivatives up to order $\lfloor \alpha\rfloor\in\mathbb{N}$ and with globally $\{\alpha\}$-H\"older continuous derivatives of order $\lfloor \alpha\rfloor$.

We denote by $C_t=C([0,T];\RR^d)$ the path space of continuous functions on $[0,T]$, endowed with the supremum norm $\| \varphi\|_{C_t}=\sup_{t\in [0,T]} |\varphi_t|$.

Given a Banach space $E$ and a parameter $q\in [1,\infty]$, we denote by $L^q_t E = L^q(0,T;E)$ the space of measurable functions $f:[0,T]\to E$ such that
\[
\| \varphi\|_{L^q_t E}:=\Big( \int_0^T \| \varphi_t\|_E^q \dd t\Big)^{\frac{1}{q}} <\infty
\]
with the usual convention of the essential supremum norm in the case $q=\infty$. Similarly, given a probability space $(\Omega,\mathcal{F},\mathbb{P})$ and $m\in [1,\infty)$, we denote by $L^m_\omega E = L^m(\Omega,\mathcal{F},\mathbb{P};E)$ the space of $E$-valued $\mathcal{F}$-measurable random variables $X$ such that
\[
\| X\|_{L^m_\omega E}:= \big( \mathbb{E}  \| X\|_E^m\big)^{\frac{1}{m}}<\infty
\]
where $\EE$ denotes expectation w.r.t.~$\mathbb{P}$. The above definitions can be concatenated by choosing at each step a different $E$, so that one can define $L^m_\omega C_t$, $L^q_t L^p_x$, $C_t^\alpha E$ and so on. Whenever $q=p$, we write for simplicity $L^p_{t,x}$ in place of $L^p_t L^p_x$.

  We simply drop the sub-index $t,x,\omega$ among using the aforementioned norms when the context is clear about the correspondence.

\begin{defin}[Wasserstein distance]\label{defin:wasserstein}
    For $p\geq1,$ let $\cP_p(E)$ denote the set of probability measures over $E$ 
    with finite $p$-moment, i.e. $\mu\in \cP_p(E) $ satisfies
    \[
    \int_E |x|^p\mu(\dd x)<\infty. 
    \]
     Furthermore, we define the $p$-Wasserstein distance 
    $\WW_p:\cP_p(E)\times \cP_p(E)\rightarrow \RR$ by  
    \[
    \WW_p(\mu,\nu):=\inf_{\rho\in \Pi(\mu,\nu)} 
        \left(\int_{E\times E} |x-y|^p  \rho(\dd x,\dd y)\right)^{\frac{1}{p}}, 
    \]
    where $\Pi(\mu,\nu)$ denotes  the set of all couplings of $\mu$ and $\nu$.
    For $p=1$ we have the Kantorovich--Rubenstein duality, where
    \[
    \WW_1(\mu,\nu)= \frac{1}{K}\sup_{|\varphi|_{\rm Lip}\leq K } 
    \int_E \varphi(y)(\mu(\dd y)-\nu(\dd y)). 
    \]
\end{defin}
Local non-determinism is a probabilistic type of 
roughness condition, specified by an index $\zeta_0 > 0$ 
which generalises the Hurst index for Gaussian processes. 
It guarantees the inclusion of a local time $L^w$ in 
appropriate function spaces. We use this notion of 
roughness to characterise our regularising path $w$. 
Proposition \ref{thm:localnondet} elucidates this notion further.
\begin{defin}[Local non-determinancy]
Let $\{w_t\}_{t\in [0,T]}$ be a $d$-dimensional Gaussian 
process on a filtered probability space $(\Omega,\cF,
	\{\cF_t\}_{t\in [0,T]},\PP)$. We say  $\{w_t\}_{t\in [0,T]}$ 
is \emph{$\zeta$-locally non-deterministic} if it satisfies 
\begin{equation}\label{eq:local non determinism}
\inf_{t\in[0,T]}\inf_{s\in[0,t]}\inf_{z\in\RR^d,|z|=1} 
	\frac{z^*{\rm Var}(w_t|\cF_s)z}{|t-s|^{2\zeta}}>0,
    \end{equation}
where ${\rm Var}(w_t|\cF_s)
	=\EE\big[(w_t-\EE(w_t|\cF_s))(w_t-\EE(w_t|\cF_s))^*|\cF_s\big]$, 
$\zeta$ is the {\em non-determinancy index}.    
\end{defin}

Another central tool that will be used in this article is the sewing lemma, introduced by Gubinelli in \cite{GUBINELLI200486} for rough paths equations, and later adopted in many different contexts. See, e.g., \cite[Sec. 4]{frizhairer} for a good introduction. 

\begin{defin}
    [Sewing germs] Let $E$ be a Banach space. Let $[0, T]$ be a given interval. Let $\Delta_n$ denote the $n$-th simplex of $[0, T]$. For a function $A: \Delta_2 \rightarrow \RR^d$ define the mapping $\delta A: \Delta_3 \rightarrow \RR^d$ via $
(\delta A)_{s,u,t}:= A_{s,t} - A_{s,u} - A_{u,t}$. Provided $A_{t,t} = 0$ we say that for $\alpha, \beta > 0$ we have $A \in C^{\alpha,\beta}_2(E)$ if $\|A\|_{\alpha,\beta} < \infty$ where $\|A\|_{\alpha,\beta}:=\|A\|_{\alpha}+\|\delta A\|_{\beta}$ with $\|A\|_{\alpha}:=\sup_{(s,t)\in\Delta_2}\frac{\|A\|_E}{|t-s|^\alpha}$ and $\|\delta A\|_{\beta}:=\sup_{(s,u,t)\in\Delta_3}\frac{\|(\delta A)_{s,u,t}\|_E}{|t-s|^\beta}$.

If for any sequence of partitions $(P^n
([s, t]))_n$  over $[s, t]$ whose mesh size goes to zero, the
quantity $\sum_{[u,v]\in P^n([s,t])}A_{u,v}$ converges to the same limit, 
then  we denote this limit by
\begin{align*}
(\mathcal{I}A)_{s,t}:=
    \lim_{n\rightarrow\infty}\sum_{[u,v]\in P^n([s,t])}A_{u,v} .
\end{align*}
\end{defin}
\begin{lem}
    [Sewing lemma {\cite[Lemma 4.2]{MR4174393}}] \label{lem:sewing} Let $0 < \alpha \leq 1 < \beta.$ Then for any $A \in C
^{\alpha,\beta}_2(E), (\mathcal{I}A)$ is well defined.
Moreover, denoting $(\mathcal{I}A)_t
:= (\mathcal{I}A)_{0,t}$ for $t\in[0,T]$, we have $(\mathcal{I}A) \in C^\alpha_t E$ and $(\mathcal{I}A)_0 = 0$ and for some
constant $ c > $0 depending only on $\beta$ we have
$\|(\mathcal{I}A)_t - (\mathcal{I}A)_s - A_{s,t}\|_E \leq c \|\delta A\|_\beta
|t - s|^\beta.$
We say the germ $A$ admits a sewing $(\mathcal{I}A)$ and call $\mathcal{I}$ the sewing operator.
\end{lem}

As remarked by Bechtold and Hofmanova in \cite[Remark 3.4]{bechtold}, 
in contrast to much regularisation-by-noise results today, 
there appears to be no advantage in the current 
(multiplicative noise) setting in using stochastic sewing 
arguments over deterministic ones. We therefore follow 
in using deterministic sewing setting here as well.

\section{Regularisation by noise and Approximation of DDSDEs}\label{sec:stochastic integration}
To investigate \eqref{eq:main intro} we begin with writing it in integral form as 
\begin{equation}\label{eq:transformed}
     x_t =x_0 + \int_0^t  b(s,F(\mu_s)-w_s)\dd s + \int_0^t a(s,F(\mu_s)-w_s)\dd \beta_s.  
\end{equation}
When $b$ and $a$ are singular functions --- recall that we are 
interested in choosing $b$ to be a distribution and $a$ to be 
only integrable of some order --- we need to make sure that the 
integrals appearing in \eqref{eq:transformed} make sense. We 
therefore begin with a brief recollection of the regularising 
effects obtained from the local time of sufficiently irregular 
paths, and show how this may be used to make sense of the integrals 
appearing in \eqref{eq:transformed}. 

\subsection{Regularisation through averaging}

Following \cite[Equation (3)]{gubicat}, we define the 
averaging operator along a path $w:[0,T]\rightarrow \RR^k$ as
\begin{align}\label{eq:averaging_operator}
T^w_t f(x) := \int_0^t f(x  - w_r)\dd r.
\end{align}
Define the occupation measure $\nu_t^w$ associated to $w$ by 
\begin{equation}
    \nu_t^w(A)=\lambda\{s\in [0,t]|\,w_s\in A\},\quad A\in \cB(\RR^k), 
\end{equation}
where $\lambda$ denotes the Lebesgue measure, and we denote by 
$L^{w}_t(x)$ the density associated to $\nu_t^w$ with respect to Lebesgue measure (whenever this 
exists). It follows that 
\begin{equation}\label{eq:localtime_repr}
    T^w_tf (x)= f\ast L^w_t(x). 
\end{equation}
The regularity of local times associated to stochastic processes 
has been a central topic of investigation in the field of stochastic 
analysis for many years. In more recent years, much advancement 
has been made in obtaining these regularity estimates on a joint 
time-space scale. In particular, we have the following proposition:
\begin{prop}[{\cite[Theorem 3.1]{harang2020cinfinity}}]\label{thm:localnondet}
    Assume $\zeta<\frac{2}{k}$. Let $\{w_t\}_{t\in [0,T]}$ be a $k$-dimensional Gaussian process on a filtered probability space $(\Omega,\cF,\{\cF_t\}_{t\in [0,T]},\PP)$, and suppose that it is $\zeta$-locally non-deterministic (i.e. $\{w_t\}_{t\in [0,T]}$ satisfies \eqref{eq:local non determinism}). 
    Then for almost all $\omega\in \Omega$ the associated local time $L^{w(\omega)}$ is contained in $C^\gamma_t H^{\lambda}_x$, with $\lambda \in \RR$, $\gamma > 0$ for 
    \begin{equation}\label{eq:general_localnondet_condition}
      \gamma<1-(\lambda+\frac{k}{2})\zeta.%
    \end{equation}
More precisely, $\PP$-almost surely, and for all $s\leq t\in [0,T]$ we have 
    \begin{align}
        \label{est:Local-time}
        \|L_t^{w(\omega)}-L_s^{w(\omega)}\|_{H^\lambda_x}\leq C(\omega)|t-s|^\gamma.
    \end{align}
\end{prop}
\begin{rem}\label{rem:optimalnum}
By Sobolev embedding, or by treating the averaged objects 
\eqref{eq:averaging_operator} directly (not via local 
times), it is possible to leverage sharper regularisation 
results in more refined Besov or Fourier--Lebesgue spaces \cite{gubicat}. 
These results may serve to optimise the numerology in 
(iii) of Hypothesis \ref{hyp:main} dictating the allowable 
ranges of smoothness indices $\gamma_0$, $\gamma_1$, 
and $\zeta_0$, that allows us to define integrals 
of rough integrands (by sewing, see Lemma \ref{lem:sewing}). 
We give further details in Remark \ref{rem:remark2.9}; 
however, we do not consider optimality of indices a focus of this paper.
\end{rem}

\begin{rem}
    The condition in \eqref{eq:local non determinism} is a type of \emph{local non-determinism} condition, which is widely used in connection with the analysis of local times and occupation measures, see e.g. \cite{horowitz}.
\end{rem}
\begin{ex}\label{exa:Fbm}
    The fractional Brownian motion (fBm) $\{B_t^H\}_{t\geq 0}$ on a probability space $(\Omega, \cF,\{\cF_t\}_{t\geq 0},\PP)$ is a Gaussian process with mean zero and covariance 
    \[
    \EE[B_t^H B_s^H] =\frac{1}{2} (t^{2H}+s^{2H}-|t-s|^{2H}). 
    \]
    The parameter $H\in (0,1)$ is known as the Hurst parameter and determines the auto-correlation of the process, as well as the irregularity of the sample paths. 
    The conditional variance of this process is given by 
    \[
    {\rm Var}(B_t^H|\cF_s)=(t-s)^{2H}. 
    \]
    Using this, we see that the fBm is locally non-deterministic in the sense of condition \eqref{eq:local non determinism} with $\zeta=H$. Thus, for almost all $\omega\in \Omega$, the local time $L^{B^H(\omega)} $ associated to the fBm is contained in $C^\gamma H^\lambda$, where 
    \[
     \lambda<\frac{1}{2H}-\frac{k}{2},\quad \mathrm{and} \quad \gamma<1-(\lambda+\frac{k}{2})H.
    \]
\end{ex}
The following corollary then follows from a simple application of Young's convolution inequality in Sobolev spaces (see, e.g., \cite[Lemma 1.4]{Bahouri2011}):

\begin{cor}\label{thm:localnondet_cor}
    Suppose $g\in H^{\alpha}_x$ and $L^w\in C^\gamma_t H^{\lambda}_x$ 
    with $\alpha,\lambda\in\RR, \gamma>0$.
    Then 
    $g\ast L^w\in C^\gamma_tW^{\alpha+\lambda,\infty}_x$, 
    and for any $0\leq s\leq t\leq T$
    \begin{equation}
        \|g\ast L_{s,t}^w\|_{W^{\alpha+\lambda, \infty}_x}
        \lesssim \|g\|_{H^\alpha_x}\|L^w\|_{C^\gamma_tH^\lambda_x}|t-s|^\gamma.
    \end{equation}
\end{cor}

\subsection{Stochastic integration}

An adapted continuous process $(x_t)_{t\in[0,T]}$ is a strong 
solution of the McKean--Vlasov equation 
\begin{align}\label{eq:mksde2}
x_t = x_0 + \int_0^t V(s, \mathcal{L}_{x_s})\dd s 
    + \int_0^t \sigma(s,\mathcal{L}_{x_s})\dd \beta_s
\end{align}
where $\mathcal{L}_{x(s)}:=Law(x_s)$, 
if the foregoing holds $\mathbb{P}$-a.s., and
$$
\EE \int_0^t |V(s,
\mathcal{L}_{x_s})| + |\sigma(s,\mathcal{L}_{x_s})|^2\dd s< \infty.
$$

An adaptation of a finite dimensional well-posedness theorem of
\cite[Theorem 1.2]{krylov1999kolmogorov} to McKean--Vlasov 
SDEs is given in \cite[Theorem 2.1]{10.1214/23-AAP2016} in 
terms of a variational distance on $\mathcal{P}_2(\RR^n)$. This 
variational distance was an adaptation of a distance-weighted 
total variation distance used, e.g., in \cite[Theorem 6.15]{villani_springer}.
A further extension \cite[Theorem 3.1]{10.1214/23-AAP2016} 
to the infinite dimensional setting adapted the monotonicty 
framework of Prev\^ot--Liu--R\"ockner \cite{PrevotRockner,LiuRockner}. 
As  \cite[Theorem 3.1]{10.1214/23-AAP2016}  was stated in the 
better known Wasserstein distance, we adapt its hypotheses 
in our well-posedness result for approximating SDEs below 
for easier reference on the readers' part.
 
Since our equation \eqref{eq:main intro} has coefficients 
independent of the process $x$ itself, requisite monotonocity 
and coercivity conditions are partially irrelevant in the 
present context. The conditions on $(V,\sigma)$ for the existence 
(and uniqueness) of strong solutions to \eqref{eq:main intro} are given in 
\cite[Section 3.1]{10.1214/23-AAP2016} in this simplified context as follows:
\begin{hyp}\label{hyp:mkv-sde_gwp}{\,}

\begin{itemize}
\item[(i)] For every $t \ge 0$, $V(t, \cdot)$, $\sigma(t,\cdot)$ 
    are continuous {\em as functions on the space of probability measures 
    with second moments}, i.e. continuous on $\mathcal{P}_2(\RR^n)$.
\item[(ii)] There exists a constant $C > 0$ 

    such that for any $t \ge 0$,
    and $\mu, \nu \in \mathcal{P}_\kappa(\RR^n)$, $\kappa \ge 2$, 
    \begin{align*}
       |V(t,\mu) - V(t,\nu)|^2 &+ 
        |\sigma(t,\mu) - \sigma(t,\nu)|^2 \\
        &\le C\big(1 + \mu(|\cdot|^\kappa) + \nu(|\cdot|^\kappa)\big)
	\WW_2^2(\mu,\nu).
    \end{align*}

\item[(iii)] There exists a $K \in L_t^1$ such that 
	for any $t \ge 0$, $x \in \RR^n$, $\mu \in \mathcal{P}_2(\RR^n)$, 
    $$
    V(t,\mu) \cdot x \le K(t) \big(1 + |x|^2 + \mu(|\cdot|^2)\big).
    $$
\item[(iv)] For any $t \ge 0$, $x \in \RR^n$, and 
    $\mu \in \mathcal{P}_\kappa(\RR^n)$, $\kappa \ge 2$, 
    with the same $K$ as in (iii),
    \begin{align*}
    |V(t,\mu)|^2 \le K(t)\big(1 + \mu(|\cdot|^\kappa)\big), 
    \qquad
    |\sigma(t,\mu)|^2 \le K(t)\big(1 + \mu(|\cdot|^2)\big).
    \end{align*}
\end{itemize}
   
\end{hyp}

We consider \eqref{eq:mksde2} with $V(s,\mu) 
    = b(s, F(\mu) - w_s)$ 
and $\sigma(s,\mu) = a(s, F(\mu) - w_s)$. 
We now construct approximations $(a_\ep, b_\ep)$ of 
$(a,b)$ so that the following McKean--Vlasov SDE has 
a unique strong solution: 
\begin{align}\label{eq:mckean-vlasov}
\dd x^\ep = b_\ep(s,F(\mu^\ep_s) - w_s) \dd t 
    + a_\ep(s, F(\mu^\ep_s) - w_s)\dd \beta_s
\end{align}
by choosing 
$$
\sigma_\ep(t,\mu) := a_\ep(t,F(\mu) - w_t),
\qquad
V_\ep(t,\mu) := b_\ep(t, F(\mu) - w_t),
$$
where $a_\ep(t,z) := a(t,\cdot)*J_\ep(z)$, and 
$b_\ep(t,z) := b(t,\cdot)*J_\ep(z)$, and $J_\ep$ is a Friedrichs mollifier 
on $\RR^k$. We can then verify conditions (i) -- (iv) 
for $V_\ep$ and $\sigma_\ep$.
\begin{lem}\label{lem:xepsilon}
For any $\ep > 0$, there exists a unique strong solution to the 
McKean--Vlasov equation \eqref{eq:mckean-vlasov}.
\end{lem}

\begin{proof}
We verify conditions (i) to (iv) in Hypothesis \ref{hyp:mkv-sde_gwp}. 
Let  $K(t) = K:=C(\|a\|_{L^\infty_tL^p_x}^2 
    +\|b\|_{L^\infty_tL_x^p(\RR^k)})$\footnote{To be noticed, here  $\|\cdot\|_{L^\infty_tL^p_x}:=\|\cdot\|_{L^\infty_tL^p(\RR^k)}$. }, for any $t\in[0,T]$. 
 For $\frac{1}{q}:=1-\frac{1}{p}$, then 
$|\pd_y a_\ep(t, y)| = \big|\big(\pd_y J_\ep* a(t,\cdot) \big)(y)\big|
\le \|\pd_x J_\ep\|_{L^q_x}\|a\|_{L^\infty_t L^p_x}$ from Young's inequality for convolution, and therefore 
$a_\ep$ is globally Lipschitz in its second entry 
(with $\ep$-dependent Lipschitz constant). For $\mu, \nu 
	\in \mathcal{P}_2(\RR^n)$,  by the Lipschitz condition on 
$F$ from Hypothesis \ref{hyp:main} and \eqref{eq:1wasserstein_use}, 
\begin{align*}
|\sigma_\ep(t,\mu) - \sigma_\ep(t,\nu)|^2
& \le \|\pd_x J_\ep * a(t,\cdot)\|_{L^\infty_x}^2 \big| F(\mu) - F(\nu)\big|^2
\le KC_\ep \WW_1^2(\mu,\nu).
\end{align*}
Since $\WW_1(\mu,\nu) \le \WW_2(\mu,\nu)$, this calculation
verifies (ii), as well as (i) of Hypothesis \ref{hyp:mkv-sde_gwp} for $\sigma$. The same 
calculation with $V$ in place of $\sigma$ will verify (i) for $V$.

Again, by Young's convolution inequality, 
\begin{align*}
|V_\ep(t,\mu) \cdot x| 
     = |b_\ep( t, F(\mu)- w_t)\cdot x| 
     \le C_\ep
     \|b\|_{L^\infty_t L^p_x} \big(1 + |x|^2\big)\le C_{\ep}
     K \big(1 + |x|^2\big).
\end{align*}
This verifies (iii) of Hypothesis \ref{hyp:mkv-sde_gwp} .

Furthermore, for (iv), we have 
\begin{align*}
    |\sigma_\ep(t,\mu)|^2
    &\le \|J_\ep\|_{L^q_x}^2 \|a\|_{L^\infty_t L^p_x}^2 \lesssim 1, \quad
 |V_\ep(t,\mu) |^2 
     &=\|J_\ep\|_{L^q_x}^2 \|b\|_{L^\infty_t L^p_x}^2  \lesssim 1.
\end{align*}
Therefore, we have the lemma using \cite[Theorem 3.1]{10.1214/23-AAP2016}.
\end{proof}

In the following we show that under the regularisation of path $w$ the terms appearing in \eqref{eq:transformed} are well-defined, even though $b$ and $a$ are very singular. The idea is to apply the sewing lemma, Lemma \ref{lem:sewing}. 
\begin{lem}\label{thm:sewing1}

Fix $p \ge 2$. Assume Hypothesis  \ref{hyp:main} holds for some $p$, 
$\gamma_0$ and  $\zeta_0$. Let $\gamma_1 \in (0,1)$ satisfying
\eqref{con.lem-gamme1}.
Finally, let $X$ be a process for which 
$$
\sup_{s \not= t} \frac{\EE \abs{X_{s,t}}^p}
    {\abs{s - t}^{p \gamma_1/2}} < \infty. 
$$
Then the germs defined for $(s,t)\in\Delta_2$
\begin{align*}
(A_1)_{s,t} &:=  \int_s^t b (s,F(\mu_s) - w_r) \dd r 
    = b(s,\cdot)\ast L^w_{s,t}(F(\mu_s))\\
(A_2)_{s,t} &:= \int_s^t  \sum_{i,j} 
    a_{ij}^2(s, F(\mu_s) - w_r) \dd r 
    = a^2_{ij}(s,\cdot)\ast L^w_{s,t}(F(\mu_s)) 
\end{align*}
respectively admit sewings $\mathcal{I}A_1$ and $\mathcal{I} A_2$, 
and for $i = 1,2$, $ t \in [0,T]$, $(\mathcal{I}A_i)_{0,t} = (A_i)_{0,t}$.
\end{lem}

\begin{proof}

Let us focus on $A_1$ as the argument for $A_2$ can be made  
along similar lines. We can  estimate the coboundary: for any $(s,\tau,t)\in\Delta_3$
\begin{equation}\label{eq:A_1_coboundary}
\begin{aligned}
\big|(\delta A_1)_{s,\tau,t}\big|
&\le \bigg|\int_\tau^t b(s,F(\mu_s) - w_r) 
    - b(\tau,F(\mu_\tau) - w_r) \dd r\bigg|\\
&=  \big| b(s)*L^w_{\tau,t}(F(\mu_s)) - b(\tau)*L^w_{\tau,t}(F(\mu_s)) \\
&\qquad\quad\,\, + b(\tau)*L^w_{\tau,t}(F(\mu_s)) 
    - b(\tau)*L^w_{\tau,t}(F(\mu_\tau)) \big|\\
&\le  \|b(s, \cdot) - b(\tau, \cdot)\|_{H^{-1}_x}\|L^w_{\tau,t}\|_{H^1_x} \\
&\qquad\quad\,\, + \|b(\tau)*L^w_{\tau,t}\|_{W^{1,\infty}_x} 
    \big|F(\mu_s) - F(\mu_\tau)\big|.
\end{aligned}
\end{equation}

Using the Lipschitz bound \eqref{eq:1wasserstein_use} on $F$, we get 
\begin{equation}\label{eq:Lipschitz_Ef1}
\begin{aligned}
\big|F(\mu_s) - F(\mu_\tau)\big|
&\le |F|_{\rm Lip} \EE |X_{s,\tau}|\\
&\le |F|_{\rm Lip} \frac{ \EE |X_{s,\tau}|}{|s - \tau|^{ \gamma_1/2}}|\tau - s|^{ \gamma_1/2} \lesssim |\tau - s|^{ \gamma_1/2}.
\end{aligned}
\end{equation}
By Young's convolution inequality and Corollary \ref{thm:localnondet_cor}, 
for $\gamma < 1 - (1 + k/2) \zeta_0$. 
\begin{align}\label{eq:Lipschitz_Ef2}
\|b(\tau)*L^w_{t,\tau}\|_{W^{1,\infty}_x} 
\le \|b(\tau)\|_{L^2_x}\|L^w_{t,\tau}\|_{H^1_x}
\lesssim \|b(\tau)\|_{L^2_x} \|L^w\|_{C^\gamma_t H^1_x}|t - \tau|^\gamma.
\end{align}
Additionally using Hypothesis \ref{hyp:main} on the 
continuity of $b(s)$ in $H^{-1}_x$,
\begin{align}\label{eq:Lipschitz_Ef3}
\|b(s, \cdot) - b(\tau, \cdot)\|_{H^{-1}_x}\|L^w_{\tau,t}\|_{H^1_x} 
\le |\tau - s|^{\gamma_0} \|L^w\|_{C^\gamma_t H^1_x}|t - \tau|^{\gamma}.
\end{align}
By the assumption $\gamma_0\wedge \frac{1}{2} > (1 + k/2)\zeta_0$ in Hypothesis \ref{hyp:main} 
together with the assumption $\gamma_1/2 > (1 + k/2)\zeta_0$ from \eqref{con.lem-gamme1}, 
we are able to choose $\gamma$ in the appropriate range so that
$\gamma \in (1 - (\gamma_1/2 \wedge \gamma_0), 1 - (1 +k/2)\zeta_0)$. 
This ensures that $\gamma_1/2 + \gamma, \gamma_0 + \gamma > 1$.
Hence upon inserting all the preceding estimates back into 
\eqref{eq:A_1_coboundary},
$$
\big|(\delta A_1)_{s,\tau,t}\big| = o(|s - t|).
$$
The standard sewing lemma then ensures a sewing.

We now identify this sewing. The alternative germ 
$$
(\tilde{A}_1)_{s,t} =  \int_s^t b(r, F(\mu_r) - w_r) \dd r
$$
(with $F(\mu)$ evaluated at $r$ instead of $s$) trivially 
admits itself as a sewing, since $(\delta \tilde{A}_1)_{s,u,t} \equiv 0$.
The difference between $A_1$ and $\tilde{A}_1$ is given 
similarly as in \eqref{eq:A_1_coboundary} by 
\begin{align*}
\big|(A_1)_{s,t} - (\tilde{A}_1)_{s,t} \big|
&=\bigg| \int_s^t  b(r,F(\mu_s) - w_r) 
    - b(r,F(\mu_r) - w_r) \dd r\bigg| \\
&\lesssim |t - s|^{1 + \gamma_1/2}.
\end{align*}

By the sewing lemma, 
there exists a sewing $\mathcal{I}A_1$ such that 
$\big|\mathcal{I} A_1)_{s,t} - (A_1)_{s,t}\big|
    \lesssim \big|(\delta A_1)_{s,u,t}\big|$.
Therefore, 
\begin{align*}
 \big|(\mathcal{I} A_1)_{s,t} - (\tilde{A}_1)_{s,t}\big|
&\le \big| (\mathcal{I} A_1)_{s,t} - (A_1)_{s,t}\big| 
    +  \big|(A_1)_{s,t} - (\tilde{A}_1)_{s,t}\big| \\
& \lesssim |t - s|^{1 + \gamma_1/2}.
\end{align*}

This implies that for any fixed $s$ (in particular, 
$s = 0$), $(\mathcal{I} A_1)_{s,t} = (\tilde{A}_1)_{s,t}$. 

\end{proof}
\begin{rem}\label{rem:remark2.9}
In \eqref{eq:A_1_coboundary} and subsequently, we chose to bound $b*L^w_{\tau,t}$ 
in $W^{1,\infty}$, given Hypothesis \ref{hyp:main} on various indices 
$\gamma_0$, $\gamma_1$, and $\zeta_0$. As alluded to in Remark 
\ref{rem:optimalnum} it would have been possible also to 
estimate $b*L^w_{\tau,t}$ in $W^{\alpha, \infty}$ for an $\alpha \in (0,1)$. 
We discuss how this changes the numerology vis-\`a-vis $\zeta_0$, $\gamma_0$, 
and $\gamma_1$ here. 

By estimating $b*L^w_{\tau,t}$ in $W^{\lambda, \infty}$, by Proposition 
\ref{thm:localnondet}, we get
$$
\|b*L^w_{\tau,t}\|_{W^{\lambda, \infty}_x} 
\le \|b\|_{L^2_x}\|L^w_{s,t}\|_{H^{\lambda}_x}
\le \|b\|_{L^2_x}\|L^w\|_{C^\gamma_tH^{\lambda}_x}  |t - s|^\gamma
$$
in place of \eqref{eq:Lipschitz_Ef2}.
In order for $L^w$ to be bounded in $C^\gamma_t H^\lambda_x$, we would 
require \eqref{eq:general_localnondet_condition} in place of 
$\gamma > \big(1 + \frac{k}2\big)\zeta_0$. On the other hand, 
the bound on \eqref{eq:A_1_coboundary} becomes 
\begin{align*}
\lesssim |t - \tau|^\gamma |\tau - s|^{\lambda\gamma_1/2} 
    + |\tau - s|^{\gamma_0}|t - \tau|^\gamma
\end{align*}
following \eqref{eq:Lipschitz_Ef1} and \eqref{eq:Lipschitz_Ef3}.
In order for the sewing argument to work, we shall then require that 
$\gamma + \lambda \gamma_1 /2, \gamma + \gamma_0 > 1$.

\end{rem}

With the above sewing results at hand, we obtain the following result for a generalised It\^o isometry as in \cite[Lemma 3.3]{bechtold}.
\begin{lem}
   Suppose $\{w_t\}_{t\in [0,T]}$ is such that  its associated local time is contained in $C^\gamma_t H^\kappa_x$ for some $\kappa>0$ and $\gamma\in (\frac{1}{2},1]$.  Then the following It\^o isometry holds 
   \begin{equation}
       \EE\left[\left( \int_s^t a^\epsilon (r, F(\mu_r)-w_r)\dd \beta_r \right)^2\right] = \|(\cI A_2^\epsilon)_{s,t}\|_{L^2(\Omega)}^2 
   \end{equation}
   where $(\cI A_2^\epsilon)_t:= \int_0^t  \sum_{i,j} 
    a_{\ep,ij}^2(s, F(\mu_s) - w_r) \dd r 
    = a^2_{\ep,ij}(s)\ast L^w_{0,t}(F(\mu_s)) $.
\end{lem}

\section{Tightness}
\label{sec:tightness}
The following result will ensure compactness of the law of approximating solutions $\mu^\epsilon$ of \eqref{eq:mckean-vlasov}
and simultaneously verify the conditions of Lemma \ref{thm:sewing1}.
\begin{lem}\label{thm:Cgamma_0_tightness}
Suppose Hypothesis \ref{hyp:main} holds for some 
$\zeta_0,\gamma_0$. Let $x^\ep$ be the unique strong solution to the SDE 
\eqref{eq:mckean-vlasov}. Then

 for any $p\ge 0$, there exists 
 $\gamma_1>0$ satisfying  $\frac{\gamma_1}2 > \Big(1 + \frac{k}2\big) \zeta_0$ so that uniformly on $\epsilon>0$,
\begin{align}\label{eq:xep_tightness}
\sup_{s \not= t} \EE \frac{ \abs{ x^\ep_{s,t}}^p}
    {\abs{t - s}^{p\gamma_1/2}} 
\lesssim_p 1, \quad 
\EE \abs{ x^\ep_t}^p \lesssim_p 1.
\end{align}
In particular, there is a version of $x^\ep$ with a.s.~$C^{\gamma_1/2 -}$ paths.
\end{lem}

\begin{proof}
By the Burkholder--Davis--Gundy inequality,
\begin{align}\label{eq:xep_cont}
\EE \abs{x^\ep_{s,t}}^p
	& \lesssim_p 
     \abs{I_1}^p +  \abs{I_2}^{p/2}, 
\end{align}
where
\begin{equation}\label{eq:I_defin1}
\begin{aligned}
I_1 &:=  \int_s^t b_\ep(r,  F(\mu^\ep_r) - w_r) \dd r,\quad
I_2 := \int_s^t  \sum_{i,j} 
    a_{\ep,ij}^2(r, F(\mu^\ep_r) - w_r) \dd r.
\end{aligned}
\end{equation}
Since $|b_\ep| \lesssim \ep^{-1}$ and $|a_\ep^2|\lesssim \ep^{-2}$, 
the quantity
$$
c_{p,\ep,\gamma_1} := \sup_{s \not= t}
    \EE \frac{\abs{x^\ep_{s,t}}^p}{|s - t|^{p \gamma_1/2}} 
$$
is therefore bounded {\em a priori} for fixed $\ep> 0$ and $\gamma_1 < 1$, 
and it remains to prove the uniformity of the bound in $\ep$.

Following \cite[Lemma 4.2]{bechtold}, we interpret the integrals 
using the sewing lemma by considering the germs: $(s,t)\in\Delta_2$,
\begin{align*}
&(A^\ep_1)_{s,t} := \int_s^t b_\ep(s, F(\mu^\ep_s) - w_r) \dd r,\\
&(A^\ep_2)_{s,t} := \int_s^t  \sum_{i,j} 
    a_{\ep,ij}^2(s, F(\mu^\ep_s) - w_r) \dd r.
\end{align*}

Let us first focus on $A^\ep_1$.
Using the smoothing operator $T^{w}$ and 
the local time $L^{w}$ of $w$, following 
\eqref{eq:localtime_repr}, we can write $A^\ep_1$ as 
\begin{align*}
(A^\ep_1)_{s,t} = \big(T^{w}_{s,t} b(s)
\big)(F(\mu_s)) 
= \big(b(s) * L^{w}_{s,t}\big)(F(\mu_s)).
\end{align*}
By Proposition \ref{thm:localnondet} the local time 
satisfies the bound $\|L_{s,t}^w\|_{L^2_x} 
    \lesssim \abs{t - s}^{\gamma}$ 
for every $\gamma < 1 - k \zeta_0/2$, therefore 
by Young's convolution inequality, uniformly in $\ep$, 
\begin{align*}
\big|(A^\ep_1)_{s,t}\big|
&\leq \big|b_\ep(s
) * L^{w}_{s,t}(F(\mu_s^\ep))\big|
\lesssim \|b(s)\|_{L^2_x} 
	\|L^{w}\|_{C^\gamma_t L^2_x} |t - s|^{\gamma}.
\end{align*}
Similarly, 
\begin{align*}
\big|(A^\ep_2)_{s,t}\big|
&\leq \big|a^2_\ep(s
) * L^{w}_{s,t}(F(\mu_s^\ep))\big|
\lesssim \|a^2(s)\|_{L^2_x} 
	\|L^{w}\|_{C^\gamma_t L^2_x} |t - s|^{\gamma}.
\end{align*}
Notice that this choice of $\gamma$ is compatible with 
$\gamma_1/2 > (1 + k/2) \zeta_0$ as long as
$$
2\Big(1 + \frac{k}2\Big) \zeta_0 
< 1 - \frac{k}2 \zeta_0 \quad i.e. \quad \Big(2 + \frac{3k}2\Big) \zeta_0 < 1, 
$$
which is exactly the first condition in Hypothesis \ref{hyp:main} (iii).
As in \eqref{eq:A_1_coboundary}, uniformly in $\epsilon$
\begin{align*}
\big|(\delta A^\ep_1)_{s,u,t}\big|
& \le \|b(s, \cdot) - b(\tau, \cdot)\|_{H^{-1}_x}\|L^w_{\tau,t}\|_{H^1_x} \\
&\quad\,\, + \|b(\tau)*L^w_{\tau,t}\|_{W^{1,\infty}_x} 
    \big|F(\mu_s^\ep) - F(\mu_\tau^\ep)\big|\\
&\le 
|\tau - s|^{\gamma_0} \|L^w\|_{C^\gamma_t L^2_x}|t - \tau|^{\gamma}\\
&\quad\,\, 
+ \|b(\tau)\|_{L^2_x} \|L^w\|_{C^\gamma_t H^1_x}
|F|_{\rm Lip} \frac{ \EE |x_{s,\tau}^\ep|}{|s - \tau|^{ \gamma_1/2}}
|t - \tau|^\gamma|\tau - s|^{ \gamma_1/2}\\
& \lesssim |t - s|^{\gamma + \gamma_0} + c_{1,\ep,\gamma_1} |t - s|^{\gamma + \gamma_1/2},
\end{align*}
with $\gamma \in (1 - (\gamma_1/2 \wedge \gamma_0), 1 - (1 - k/2)\zeta_0)$,
which ensures $\gamma_1/2 + \gamma, \gamma_0 + \gamma > 1$.

Therefore $A^\ep_1$ admits a sewing $\mathcal{I}A^\ep_1$ and uniformly on $\epsilon$
\begin{align}\label{eq:IG1_cont}
 \big|\mathcal{I}A^\ep_1\big|^p \lesssim \big|(A^\ep_1)_{s,t}\big|^p + 
    |t - s|^{p(\gamma + \gamma_0)} 
    + c_{p,\ep,\gamma_1} |t - s|^{p(\gamma + \gamma_1/2)}.
\end{align}
Likewise, $A^\ep_2$ admits a sewing $\mathcal{I} A^\ep_2$
\begin{align}\label{eq:IG2_cont}
 \big|\mathcal{I}A^\ep_2\big|^p \lesssim 
 \big|(A^\ep_2)_{s,t}\big|^p + 
     |t - s|^{p(\gamma + \gamma_0)} 
     + c_{p,\ep,\gamma_1} |t - s|^{p(\gamma + \gamma_1/2)}.
\end{align}

We interpret the integrals $I_1$ and $I_2$ in 
\eqref{eq:I_defin1}  respectively by the sewings $\mathcal{I}A^\ep_1$ and $\mathcal{I}A^\ep_2$ using Lemma \ref{thm:sewing1}. 
Inserting \eqref{eq:IG1_cont} and \eqref{eq:IG2_cont} 
into \eqref{eq:xep_cont}, we find uniformly on $\epsilon$
\begin{align*}
\EE \abs{x^\ep_{s,t}}^p &\lesssim_p |t - s|^{p\gamma_1/2} + 
 |t - s|^{p(\gamma + \gamma_0)/2} \\
    &\quad\,\, 
    + c_{p,\ep,\gamma_1} |t - s|^{p(\gamma + \gamma_1/2)}
    + c_{p,\ep,\gamma_1}^{1/2} |t - s|^{p(\gamma + \gamma_1/2)/2}\\
&\lesssim_p |t - s|^{p\gamma_1/2} +  c_{p,\ep,\gamma_1} |t - s|^{p(\gamma + \gamma_1/2)}\\
&\quad\,\,
    + c_{p,\ep,\gamma_1}^{1/2} |t - s|^{p\gamma_1/2} T^{p(\gamma - \gamma_1/2)/2},
\end{align*}
where we used $\gamma > 1 - \gamma_1/2$, which implies 
$\gamma - \gamma_1/2 > 1 - \gamma_1 > 0$. 
This allows us to deduce 
$$
\EE \frac{\big|x^\ep_{s,t}\big|^p}{|t - s|^{p\gamma_1/2}}
\lesssim_{p,T} 
1 + c_{p,\ep,\gamma_1} |t - s|^{p\gamma} 
+ c_{p,\ep,\gamma_1}^{1/2}.
$$
By taking a sequence $(s,t) = (s_n,t_n)$ where $|s_n - t_n| \to 0$ 
to approximate $c_{p,\ep,\gamma_1}$ on the left hand side, we 
arrive at the bound $ c_{p,\ep,\gamma_1} \lesssim 1 
+ c_{p,\ep,\gamma_1}^{1/2}$. This shows that $c_{p,\ep,\gamma_1}$ 
is $\ep$-independent.
Then taking $p$ to be arbitrarily big, by Kolmogorov's continuity 
criterion, we conclude that there is a version of $x^\ep$ 
with a.s.~$C^{\gamma_1/2}$ paths, and satisfies the bound 
\eqref{eq:xep_tightness} as sought. 

\end{proof}

We immediately attain by the Skorokhod representation 
theorem the following limiting result on a new probability space:
\begin{prop}[Skorokhod representation theorem]\label{thm:skorokhod}
Let $\mathcal{X} := \RR \times C^{\gamma_1/4}([0,T]) \times C([0,T])$, 
with $\gamma_1 \in (0,1)$, satisfying
$$
\frac{\gamma_1}2 > \Big(1 + \frac{k}2\big) \zeta_0.
$$
There exists a probability space 
$\tilde{E}:= (\tilde{\Omega}, \tilde{\mathcal{F}}, 
    \tilde{\mathbb{P}})$ and $\mathcal{X}$-valued random variables 
$\{\tilde{X}_k :=(\tilde{x}_0^k, \tilde{x}^k, \tilde{\beta}_k)\}_{k = 1}^\infty$
and $\tilde{X} := (\tilde{x}_0, \tilde{x}, \tilde{\beta})$ such that 
along a subsequence $\ep_k \downarrow 0$, 
$$
(x^{\ep_k}_0,  x^{\ep_k}, \beta) \sim \tilde{X}_k, \qquad
\tilde{X}_k \xrightarrow{k\uparrow \infty} \tilde X \,\, \text{ in } \mathcal{X}.
$$
\end{prop}

Let $\mathcal{N}$ denote the collection of $\tilde{\mathbb{P}}$-null sets. 
We conclude this section by equipping the probability 
space $\tilde{E}$ established in Proposition \ref{thm:skorokhod} 
with a sequence of filtrations
\begin{align}\label{eq:filtrations}
\tilde{\mathcal{F}}^k_t := \Sigma(\{ \tilde{X}_k(t): s \in [0,t]\} \cup \mathcal{N}),
\end{align}
where $\Sigma(R)$ denotes the $\sigma$-algebra generated by $R$. 

Since we are concerned about the convergence of the distribution of 
$x^{\ep_k}(t)$ for any fixed $t$, we derive the convergence of the 
laws at fixed $t \in [0,T]$ from the a.s.~convergence given in Proposition 
\ref{thm:skorokhod} above. This will be central for our construction of a limiting solution in the subsequent section. 

\begin{lem}\label{thm:law_convergence}
Let $\gamma_1 > 0$ and let $\{y_k\} \subset_b L^p(\Omega;C^{\gamma_1}([0,T]))$, $p > 1$, be a sequence that 
that tends to $y$ a.s.~in $C_t$. For any $t \in [0,T]$, let 
$\mu^k_t$ be the law of $y_k(t)$ (on $\RR^n$) and $\mu_t$ be the law of $y(t)$ on $\RR^n$. 
\begin{itemize}
\item[(i)] For any fixed $k$, 
$\WW_1(\mu^k_t,\mu^k_s) \le |t - s|^{\gamma_1}$.
\item[(ii)] 
    Uniformly in $t \in [0,T]$, 
    $\WW_1(\mu^k_t,\mu_t) \to 0$ as $k \uparrow \infty$.
\end{itemize}
\end{lem}

\begin{proof}
The first statement follows directly from Kantorich duality for the $1$-Wasserstein 
norm. For any $s, t \in [0,T]$, we have
\begin{align*}
\WW_1(\mu^k_t, \mu^k_s) 
&= \sup_{|g|_{\rm Lip} \le 1} \int_{\RR^n}g(z) \big(\mu^k_t - \mu^k_s\big)(\dd z)\\
&= \sup_{|g|_{\rm Lip} \le 1} \EE \big[g(y^k_t) - g(y^k_s)\big]\\
& \le \sup_{|g|_{\rm Lip} \le 1} |g|_{\rm Lip}\, \EE |y^k_t - y^k_s| 
\lesssim  |t - s|^{\gamma_1}.
\end{align*}

The second statement follows from a similar calculation, 
whereby
\begin{align*}
\WW_1(\mu^k_t, \mu_t) \le \sup_{|g|_{\rm Lip} \le 1} 
    |g|_{\rm Lip} \,\EE \big|y^k_t - y_t\big|.
\end{align*}
We first show that $y_t \in L^{p - \epsilon}(\Omega; C_t)$ for some 
$p - \epsilon > 1$. Using the assumed a.s.~convergence, 
$ \big|y^k_t - y_t\big| \to 0$ as $k \uparrow \infty$, uniformly in $t$. 
Then using the boundedness of  $\{y_k\} \subset_b L^p(\Omega;C^{\gamma_1}_t)$,
Vitali's convergence theorem implies the convergence of 
$\EE \Big|\big|y^k_t \big| 
    - | y_t|\Big|^{p - \epsilon}$ for any $p - \epsilon > 1$, 
and hence $\EE |y_t|^{p - \epsilon} < \infty$ by the triangle inequality. 
Jensen's inequality now implies the 
convergence $\EE \big|y^k_t - y_t\big|^{p'}$ for any 
$1 \le p' < p - \epsilon$. This proves (ii).

\end{proof}

\section{Identification of the limit}
\label{sec:lim}
Recall the representatives $\tilde{X}^k$ defined in 
Proposition \ref{thm:skorokhod}, and their laws $\mu^k_s$.
For any fixed $k$, by the equality of laws, 
\begin{align*}
\tilde{x}^k_t = \tilde{x}^k_0 
    + \int_0^t b_{\ep_k}(s, F(\mu^k_s) - w_s) \dd t
    + \int_0^t a_{\ep_k}(s, F(\mu^k_s) - w_s) \dd \tilde{\beta}^k_s.
\end{align*}

We now consider the limit as $k \uparrow \infty$, 
which can be taken using the convergence asserted 
in Proposition \ref{thm:skorokhod} and a standard
martingale identification argument.

Consider the processes
\begin{equation}\label{eq:martingale_approx}
\begin{aligned}
&\tilde{M}^k(t) := \tilde{x}^k_t - \tilde{x}^k_0 
    - \int_0^t b_{\ep_k}(s,  F(\mu^k_s) - w_s)\dd s,\\
&\tilde{R}^k(t):= \big|\tilde{M}^k(t)\big|^2 - \int_0^t a^2_{\ep_k}(s, 
          F(\mu^k_s) - w_s)\dd s, \,\textrm{and}\\
&\tilde{N}^k(t) := \tilde{M}^k(t) \tilde{\beta}^k_t -  \int_0^t a_{\ep_k}(s, 
          F(\mu^k_s) - w_s)\dd s.
\end{aligned}
\end{equation}
We interpret the time integrals as sewings:
\begin{align*}
\int_0^t b_{\ep_k}(s, F(\mu^k_s) - w_s)\dd s, \qquad
\int_0^t a^2_{\ep_k}(s, F(\mu^k_s) - w_s)\dd s
\end{align*}
as in \eqref{eq:IG1_cont} and \eqref{eq:IG2_cont}.
Similarly, we can interpret the remaining integral 
as a sewing:
\begin{align*}
 \int_0^t a_{\ep_k}(s, F(\mu^k_s) - w_s)\dd s
    & = \mathcal{I} \int_s^{s'} a_{\ep_k}(s, F(\mu^k_s) - w_r)\dd r.
\end{align*}

We have the following martingale property for the processes in \eqref{eq:martingale_approx}.
\begin{lem}
For each fixed $k$, the processes $\tilde{M}^k$, $\tilde{N}^k$,
and $\tilde{R}^k$ defined in \eqref{eq:martingale_approx} are 
martingales relative to the filtration $\{\tilde{\mathcal{F}}^k_t\}_{t \in [0,T]}$ 
constructed in \eqref{eq:filtrations}.
\end{lem}

This is a standard argument that goes back to \cite{Brzezniak01092011}. 
Let $\phi$ be a bounded continuous functional on $C_t\times C_t$. 
Let $\tilde{x}^k$ and $\tilde{\beta}_k$ be as defined in Proposition 
\ref{thm:skorokhod}. The processes defined in \eqref{eq:martingale_approx} 
are martingales if and only if, for every  $0 \le s \le t \le T$, 
\begin{equation}\label{eq:martingale_n}
\begin{aligned}
&\tilde{\EE} \big[(\phi(\tilde{x}^k|_{[0,s]}, \tilde{\beta}_k|_{[0,s]}) 
    (\tilde{M}^k(t) - \tilde{M}^k(s))\big] = 0,\\
&\tilde{\EE} \Big[\phi(\tilde{x}^k|_{[0,s]}, \tilde{\beta}_k|_{[0,s]}) 
 (\tilde{R}^k(t) - \tilde{R}^k(s))\Big] = 0,\\
&\tilde{\EE} \Big[\phi(\tilde{x}^k|_{[0,s]}, \tilde{\beta}_k|_{[0,s]})
    (\tilde{N}^k(t) - \tilde{N}^k(s))\Big] = 0.
\end{aligned}
\end{equation}
This in turn is a consequence of the equivalence of laws given by 
Proposition \ref{thm:skorokhod}. We now take limits in each of the 
equations of \eqref{eq:martingale_n} to get 

\begin{lem}\label{lem:limit}
Define the processes:
\begin{align*}
&\tilde{M}(t) := \tilde{x}_t - \tilde{x}_0 
    - \int_0^t b(s, F(\mu_s) - w_s)\dd s,\\
&\tilde{R}(t):= \big|\tilde{M}(t)\big|^2 - \int_0^t a^2(s, 
         F(\mu_s) - w_s)\dd s, \,\textrm{and}\\
&\tilde{N}(t) := \tilde{M}(t) \tilde{\beta}_t -  \int_0^t a(s, 
         F(\mu_s) - w_s)\dd s.
\end{align*}
Then
\begin{equation}\label{eq:martingale_limit}
\begin{aligned}
&\tilde{\EE} \big[(\phi(\tilde{x}|_{[0,s]}, \tilde{\beta}|_{[0,s]}) 
    (\tilde{M}(t) - \tilde{M}(s))\big] = 0,\\
&\tilde{\EE} \Big[\phi(\tilde{x}|_{[0,s]}, \tilde{\beta}|_{[0,s]})
  (\tilde{R}(t) - \tilde{R}(s))\Big] = 0,\\
&\tilde{\EE} \Big[\phi(\tilde{x}|_{[0,s]}, \tilde{\beta}|_{[0,s]}) 
    (\tilde{N}(t) - \tilde{N}(s) )\Big] = 0.
\end{aligned}
\end{equation}

\end{lem}

\begin{proof}

We perform the calculations for the convergence $\tilde{M}^k 
\to \tilde{M}$ giving us the first equation of \eqref{eq:martingale_limit}. 
The remaining limits are analogous. Since the martingale property is stable under almost sure convergence combined with uniform integrability, the limiting processes \( \tilde{M} \), \( \tilde{N} \), and \( \tilde{R} \) remain martingales with respect to the limit filtration.

We focus on the integral term in $\tilde{M}^k$ defined 
by sewing. The convergence $\tilde{x}^k - \tilde{x}^k_0 
    \to \tilde{x}_t - \tilde{x}_0$ 
in $C^{\gamma_1/4}_t$, $\tilde{\PP}$-a.s.~follows 
directly from Proposition \ref{thm:skorokhod}.

By Young's convolution inequality  in $W^{1,\infty}$,
\begin{align*}
&\Big|\int_s^t b(s,F(\mu^k_s) - w_r) \dd r 
    - \int_s^t b(s,F(\mu_s) - w_r) \dd r \Big|\\
& = \big|b(s)*L^w_{s,t} (F(\mu^k_s)) -  b(s)*L^w_{s,t} 
    (F(\mu_s))\big|\\
& \le \|b(s)\|_{L^2_x} \|L^w\|_{C^\gamma_t H^1_x}  |t - s|^{\gamma}
    |F(\mu^k_s) - F(\mu_s)|,
\end{align*}
where we used Corollary \ref{thm:localnondet_cor}. 

With $\mu^k_t$ denoting the law of $\tilde{x}^k_s$ and $\mu_t$
the law of $\tilde{x}$ at $t$, uniformly in $t$,
\begin{align*}
|F(\mu^k_t) - F(\mu_t)| 
&\lesssim \WW_1(\mu^k_t, \mu_t) 
\xrightarrow{\text{Lemma \ref{thm:law_convergence} (ii)}} 0.
\end{align*}

On the other hand, by \cite[Theorem 2.5]{bechtold}, which can 
readily be derived from Lemma \ref{thm:localnondet}, 
\begin{align*}
&\bigg|\int_s^t b_{\ep_k}(s,F(\mu^k_s) - w_r) \dd r 
    - \int_s^t b(s,F(\mu^k_s) - w_r) \dd r\bigg|\\
& = \big|\big(T^{w}_{s,t} b_{\ep_k}(s)\big)(F(\mu^k_s)) 
    - \big(T^{w}_{s,t} b(s)\big)(F(\mu^k_s))\big|\\
& \le  \|L^w\|_{C^\gamma_t H^1_x} |t - s|^{\gamma}
    \underbrace{\|b_{\ep_k}(s) - b(s)\|_{H^{-1}_x}}_{= o_{k \uparrow \infty}(1)}.
\end{align*}

Finally, we must need bound the co-boundary. Setting 
$B^k_{s,t} := \int_s^t b_{\ep_k}(s,F(\mu^k_s) - w_r) \dd r$ 
and using \eqref{eq:1wasserstein_use}, Lemma 
\ref{thm:law_convergence} (i), and the bounded inclusion of the paths 
$\{\tilde{x}^k\} \subset_b L^p(\tilde{\Omega};C^{\gamma_1/2}([0,T]))$ 
guaranteed by Lemma \eqref{thm:Cgamma_0_tightness} and 
the equality of laws in Proposition \ref{thm:skorokhod}, we have $\forall (s,\tau,t)\in\Delta_3$
\begin{align*}
&\big|\delta B_{s,\tau, t}^k\big|
\\& = \big|b_{\ep_k}(s)*L^{w}_{\tau,t}  (F(\mu^k_s)) 
    - b_{\ep_k}(s)*L^{w}_{\tau,t} (F(\mu^k_\tau))\\
&\qquad + b_{\ep_k}(s)*L^{w}_{\tau,t} (F(\mu^k_\tau)) 
    - b_{\ep_k}(\tau)*L^{w}_{\tau,t} (F(\mu^k_\tau))\big|\\
&\le \|b_{\ep_k}(s)\|_{L^2_x} \|L^w_{\tau, t}\|_{H^1_x}
    |F|_{\rm Lip}\WW_1(\mu^k_s, \mu^k_t)
    + \|L^w_{\tau,t}\|_{H^1_x}
    \|b_{\ep_k}(s) - b_{\ep_k}(\tau)\|_{H^{-1}_x}\\
&\le \|b_{\ep_k}(s)\|_{L^2_x}  \|L^w\|_{C^\gamma_t H^1_x} |t - \tau|^\gamma |s - \tau|^{\gamma_1/2}
    +\|L^w\|_{C^\gamma_t H^1_x} |t - \tau|^\gamma |s - \tau|^{\gamma_0}
\end{align*}
Since $\gamma + \gamma_0$, $\gamma + \gamma_1/2 > 1$ (see, e.g.,  
Lemma \ref{thm:sewing1} above), via the sewing lemma, we have a well-defined 
integral. The convergence $\tilde{M}^k - \tilde{M} \to 0$ a.s.~in $C_t$ 
is also assured. By Vitali's convergence theorem, this a.s.~convergence 
can be upgraded to convergence in $L^1_{\tilde{\omega}} C_t$, which 
guarantees the first equation of \eqref{eq:martingale_limit}.
\end{proof}

We now are ready to prove our main theorem on existence of solutions  to DDSDEs \eqref{eq:main intro}
under the assumption of $C^2 \cap W^{1,\infty}$ initial data.

\begin{thm}\label{thm:main-1}
       If Hypothesis \ref{hyp:main} holds, then there exists a weak solution (in the probabilistic sense) to \eqref{eq:main intro}.
\end{thm}
\begin{proof}
With the detailed results from Section \ref{sec:stochastic integration}, \ref{sec:tightness}, and \ref{sec:lim} at hand, it is sufficient to outline the steps of the proof here and refer to the rigorous statement in the corresponding places.

Following from Lemma \ref{lem:xepsilon}, we know that there exists a unique solution $x^\epsilon$ to the equation \eqref{eq:mckean-vlasov} which is the approximation equation of  \eqref{eq:main intro}. 
In this way  we get a sequence of solutions $(x^\epsilon, \beta,\PP)$ ($\beta$ is a Brownian motion on the filtered probability space ($\Omega,\mathcal{F}, \PP$)) which furthermore is dense by Lemma \ref{thm:Cgamma_0_tightness}. Hence for such sequence we get from Proposition \ref{thm:skorokhod} that there exists a subsequence $(x^{\epsilon_k}, \beta^k, \tilde{\PP})$ ($\beta^k$ is some Brownian motion on the filtered probability space ($\tilde\Omega, \tilde{\mathcal{F}}, \tilde\PP$)) so that as $k\rightarrow\infty$, $\tilde X_k=(x_0^{\epsilon_k}, x^{\epsilon_k},\beta^{k}, \tilde \PP)$ converges in law to $X=(x_0, x,\tilde\beta, \tilde \PP)$, which is shown  from Lemma \ref{lem:limit} to be a weak solution to  \eqref{eq:main intro}.
\end{proof}

\newcommand{\etalchar}[1]{$^{#1}$}

\end{document}